\documentclass[10pt]{article}
\usepackage[T1]{fontenc}
\usepackage[utf8]{inputenc}
\usepackage[italian,english]{babel}
\usepackage{authblk}
\usepackage{algpseudocode}
\usepackage[ruled]{algorithm}
\usepackage{amsthm}
\usepackage{amsfonts}
\usepackage{amsmath}
\usepackage{mathtools}
\usepackage{enumerate}
\usepackage{multirow}
\usepackage{bigstrut}
\usepackage[nottoc]{tocbibind}
\usepackage[square,numbers]{natbib}

\newcommand{\R}{\mathbb R} 
\DeclarePairedDelimiter{\norm}{\lVert}{\rVert} 
\DeclarePairedDelimiter{\abs}{\lvert}{\rvert} 

\newtheorem{theorem}{Theorem}
\newtheorem{proposition}{Proposition}
\newtheorem{lemma}{Lemma}
\newtheorem{corollary}{Corollary}
\newtheorem{remark}{Remark}
\newtheorem{definition}{Definition}

\begin{document}
\pagestyle{empty}

\vskip 2cm \begin{center}{\LARGE On global minimizers of quadratic functions with cubic regularization}\end{center}
\par\bigskip
\centerline{\large Andrea Cristofari$^1$, Tayebeh Dehghan Niri$^2$, Stefano Lucidi$^3$}
\par\bigskip\bigskip

\centerline{$^1$Department of Mathematics} \centerline{University of Padua} \centerline{Via Trieste, 63, 35121 Padua, Italy}
\par\medskip
\centerline{$^2$Department of Mathematics} \centerline{Yazd University} \centerline{P.O. Box 89195-74, Yazd, Iran}
\par\medskip
\centerline{$^3$Department of Computer, Control and Management Engineering} \centerline{Sapienza University of Rome} \centerline{Via Ariosto, 25, 00185 Rome, Italy}
\par\medskip
\centerline{e-mail (Cristofari): andrea.cristofari@unipd.it} \centerline{e-mail (Dehghan Niri): t.dehghan@stu.yazd.ac.ir} \centerline{e-mail (Lucidi): lucidi@diag.uniroma1.it}

\par\bigskip\noindent \centerline{\bf Abstract}
In this paper, we analyze some theoretical properties of the problem of minimizing a quadratic function with a cubic regularization term,
arising in many methods for unconstrained and constrained optimization that have been proposed in the last years.
First we show that, given any stationary point that is not a global solution,
it is possible to compute, in closed form, a new point with a smaller objective function value.
Then, we prove that a global minimizer can be obtained by computing a finite number of stationary points.
Finally, we extend these results to the case where stationary conditions are approximately satisfied, discussing some possible algorithmic applications.

\bigskip\par\noindent

\pagestyle{plain} \setcounter{page}{1}

\section{Introduction}\label{sec:introduction}
In this paper, we address the solutions of the following (possibly non-convex) optimization problem:
\begin{equation}\label{cubic_model}
\min_{s \in \R^n}\ m(s) := c^T s + \frac 1 2 s^T Q s + \frac 1 3 \sigma \norm s^3,
\end{equation}
where $c \in \R^n$, $Q$ is a symmetric $n \times n$ matrix, $\sigma$ is a positive real number and, here and in the rest of the article, $\norm{\cdot}$ is the Euclidean norm.

In recent years, there has been a growing interest in studying the properties of problem~\eqref{cubic_model}, since functions of the form of $m(s)$
are used as local models (to be minimized) in many algorithmic frameworks for unconstrained optimization
~\cite{griewank:1981,nesterov:2006a,weiser:2007,nesterov:2008,cartis:2011a,cartis:2011b,gould:2012,bellavia:2014,benson:2014,bianconcini:2015,dussault:2015,bianconcini:2016,birgin:2017},
which have been even extended to the constrained case~\cite{nesterov:2006b,cartis:2012,benson:2014}.
To be more specific, let us consider the unconstrained optimization problem
\[
\min_{x \in \R^n}\  f(x),
\]
where $f \colon \R^n \to \R$ is a twice continuously differentiable function.
The class of methods proposed in the above cited papers is mostly characterized by the iteration $x^{k+1} = x^k + s^k$,
being $s^k$ a (possibly approximate) minimizer of the cubic model
\[
m^k(s) := f(x^k) + \nabla f(x^k)^Ts + \frac 1 2 s^T \nabla^2 f(x^k) s + \frac 1 3 \sigma^k \norm s^3,
\]
where $\sigma^k$ is a suitably chosen positive real number.
Interestingly, it can be shown that, under suitable assumptions, this algorithmic scheme is able to achieve quadratic convergence rate and
a worst-case iteration complexity better than the gradient method.
In particular, if $\nabla^2 f(x)$ is Lipschitz continuous and $s^k$ is a global minimizer of $m^k(s)$,
Nesterov and Polyak~\cite{nesterov:2006a} proved a worst-case iteration count of order $O(\epsilon^{-3/2})$ to obtain $\norm{\nabla f(x^k)} \le \epsilon$.
Cartis, Gould and Toint~\cite{cartis:2011a,cartis:2011b} generalized this result,
obtaining the same complexity bound, but allowing for a symmetric approximation of $\nabla^2 f(x^k)$ to be used in $m^k(s)$ and
relaxing the condition that $s^k$ is a global minimizer of $m^k(s)$.
Moreover, superlinear and quadratic convergence rate were proved under appropriate assumptions, but without
requiring $\nabla^2 f(x^k)$ to be globally Lipschitz continuous.

The intuition behind the algorithm proposed by Cartis, Gould and Toint is that the parameter $\sigma^k$ plays the same role as the (reciprocal of the)
trust-region radius in trust-region methods.
Moreover, some theoretical properties of trust-region models can be extended to~\eqref{cubic_model},
such as the existence of necessary and sufficient conditions for global minimizers even when $m(s)$ is non-convex~\cite{griewank:1981,nesterov:2006a,cartis:2011a}.
In this fashion, Cartis, Gould and Toint proposed the Adaptive Regularization algorithm using Cubics (ARC) that,
besides having the theoretical convergence properties mentioned above, is in practice comparable with state-of-the-art trust-region methods.

In this respect, in the above cited papers different strategies were proposed to minimize $m^k(s)$.
In particular, in~\cite{nesterov:2006a,cartis:2011a} some iterative techniques were devised to compute global minimizers,
that are based on solving a one-dimensional non-linear equation.

Starting from these considerations, here we focus on the solutions of problem~\eqref{cubic_model}, pointing out some theoretical properties
that, besides their own interest, may be useful from an algorithmic point of view.
In particular, we first extend the results obtained in~\cite{lucidi:1998} for trust-region models and we show that,
given any stationary point of~\eqref{cubic_model} that is not a global minimizer, we can compute, in closed form, a new point that reduces $m(s)$.
So, a global minimizer of~\eqref{cubic_model} can be obtained by repeating this step a finite number of times,
that is, computing at most $2(k+1)$ stationary points, where $k$ is the number of distinct negative eigenvalues of the matrix $Q$.
Further, we show how this strategy can be generalized to the case where stationary conditions are approximately satisfied,
opening to a possible practical usage of the proposed results.

The rest of the paper is organized as follows.
Section~\ref{sec:properties} is the core of the paper, where we point out some theoretical properties of the stationary points of~\eqref{cubic_model} and
analyze how to compute global minima by escaping from stationary points that are not global minimizers.
In Section~\ref{sec:practical_usage} we generalize these properties, considering approximate stationary points, and we
briefly discuss how these results can used in a more general framework.
Finally, we draw some conclusions in Section~\ref{sec:conclusions}.

\section{Properties of stationary points}\label{sec:properties}
In this section, we present the main results of the paper.
First, let us report the definition of stationary points of problem~\eqref{cubic_model} and
recall a known result on necessary and sufficient conditions for global optimality, whose proof can be found in~\cite{cartis:2011a}.
From now on, we indicate with $I$ the $n \times n$ identity matrix.
\begin{definition}
We say that $s^* \in \R^n$ is a stationary point of problem~\eqref{cubic_model} if
\[
\nabla m(s^*) = c + Qs^* + \sigma \norm{s^*} s^* = 0,
\]
or equivalently,
\begin{align}
& c + Qs^* + \lambda s^* = 0, \label{stat1} \\
& \lambda = \sigma\|s^*\|. \label{stat2}
\end{align}
\end{definition}

\begin{theorem}\label{th:opt}
A point $s^* \in \R^n$ is a global minimizer of problem~\eqref{cubic_model} if and only if
it satisfies stationary conditions~\eqref{stat1}--\eqref{stat2} and the matrix $(Q + \sigma \norm {s^*} I)$ is positive semidefinite.
Moreover, $s^*$ is unique if $(Q + \sigma \norm {s^*} I)$ is positive definite.
\end{theorem}

Now, exploiting the close relation between problem~\eqref{cubic_model} and the trust-region model (see~\cite{conn:2000} for an overview on trust-region methods),
we extend the results obtained in~\cite{lucidi:1998} to show that
\begin{enumerate}[(i)]
\item given a stationary point $\bar s$ of~\eqref{cubic_model} that is not a global minimizer,
    we can compute, in closed form, a new point $\hat s$ such that $m(\hat s) < m(\bar s)$;
\item a global minimizer of~\eqref{cubic_model} can be obtained by computing at most \mbox{$2(k+1)$} stationary points,
    where $k$ is the number of distinct negative eigenvalues of the matrix $Q$.
\end{enumerate}

We start by proving the first point, as stated in the following theorem.
\begin{theorem}\label{th:global}
Let $\bar s$ be a stationary point of problem~\eqref{cubic_model}. We define the point $\hat s$ as follows:
\begin{enumerate}[(a)]
\item if $c^T \bar s > 0$, then
    \[
    \hat s := -\bar s;
    \]
\item if $c^T \bar s \le 0$ and a vector $d \in \R^n$ exists such that $d^T(Q + \sigma \norm{\bar s} I) d < 0$,
    \begin{enumerate}[(i)]
    \item if $\bar s = 0$, then
        \[
        \hat s := \bar s + \alpha d,
        \]
        with
        \[
        0 < \alpha < -\frac{3 \, d^T Q d}{2 \, \sigma \norm d^3};
        \]
    \item if $\bar s \ne 0$ and $\bar s^T d \ne 0$, then
        \[
        \hat s := \bar s - 2 \frac{\bar s^T d}{\norm d^2} d;
        \]
    \item if $\bar s \ne 0$ and $\bar s^T d = 0$, then
        \[
        \hat s := \bar s - 2 \frac{\bar s^T z}{\norm z^2} z,
        \]
        where $z := \bar s + \alpha d$ and
        \[
        \alpha > \frac{c^T d - \sqrt{(c^T d)^2 + (c^T \bar s) \bigl[d^T (Q + \sigma \norm{\bar s} I) d\bigr]}} {d^T (Q + \sigma \norm{\bar s} I) d}.
        \]
    \end{enumerate}
\end{enumerate}
We have that
\[
m(\hat s) < m(\bar s).
\]
\end{theorem}

\begin{proof}
In case~(a), we can write
\[
\begin{split}
m(\hat s) & = m(-\bar s) = c^T (-\bar s) + \frac 1 2 \bar s^T Q \bar s + \frac 1 3 \sigma \norm{\bar s}^3 \\
          & < c^T \bar s + \frac 1 2 \bar s^T Q \bar s + \frac 1 3 \sigma \norm{\bar s}^3 = m(\bar s).
\end{split}
\]
Now, we consider case~(b) and distinguish the three subcases.
\begin{enumerate}[(i)]
\item From~\eqref{stat1}--\eqref{stat2}, we have that $c=0$. Thus, we can write
    \[
    m(\bar s + \alpha d) = m(\alpha d) = \frac 1 2 \alpha^2 d^T Q d + \frac 1 3 \sigma \alpha^3 \norm d^3, \quad \forall \alpha \in \R^n.
    \]
    Consequently,
    \[
    m(\bar s + \alpha d) = \frac 1 6 \alpha^2 ( 3 d^T Q d + 2 \sigma \alpha \norm d^3 ) < 0 = m(\bar s),
    \]
    for all $\displaystyle{0 < \alpha < -\frac{3 \, d^T Q d}{2 \, \sigma \norm d^3}}$.
\item First, we observe that
    \begin{equation}\label{proof:norm_hat_x}
    \biggl\| \bar s - 2 \frac{\bar s^T d}{\norm d^2} d \biggr\|^2
    = \norm{\bar s}^2 + \biggl(2 \frac {\bar s^T d}{\norm d^2} \biggr)^2 \norm d^2 - 4 \frac{\bar s^T d}{\norm d^2} (\bar s^T d) =  \norm{\bar s}^2.
    \end{equation}
    Moreover, the function $m(s)$ can be written as
    \begin{equation}\label{proof:m(x)}
    m(s) = c^T s + \frac 1 2 s^T (Q + \sigma \norm s I) s - \frac 1 6 \sigma \norm s^3.
    \end{equation}
    Using~\eqref{proof:norm_hat_x} and~\eqref{proof:m(x)}, we can write $\displaystyle{m\biggl(\bar s - 2 \frac{\bar s^T d}{\norm d^2} d\biggr)}$ as
    \[
    c^T \biggl(\bar s - 2 \frac{\bar s^T d}{\norm d^2} d\biggr) +
    \frac 1 2 \biggl(\bar s - 2 \frac{\bar s^T d}{\norm d^2} d\biggr)^T
    (Q + \sigma \| \bar s \| I) \biggl(\bar s - 2 \frac{\bar s^T d}{\norm d^2} d\biggr)
    - \frac 1 6 \sigma \norm{\bar s}^3.
    \]
    Rearranging and taking into account that $\nabla m(\bar s) = Q \bar s + \sigma \norm{\bar s} \bar s + c$, we obtain
    \begin{equation}\label{proof:case_b_2}
    m\biggl(\bar s - 2 \frac{\bar s^T d}{\norm d^2} d\biggr) =
    m(\bar s) + \frac 1 2 \biggl( 2 \frac{\bar s^T d}{\norm d^2} \biggr)^2 d^T (Q + \sigma \norm{\bar s} I) d
    - 2 \frac{\bar s^T d}{\norm d^2} \nabla m(\bar s)^T d.
    \end{equation}
    Stationary conditions~\eqref{stat1}--\eqref{stat2} imply that $\nabla m(\bar s) = 0$.
    Exploiting the fact that $d^T(Q + \sigma \norm{\bar s} I) d < 0$, we get
    $\displaystyle{m\biggl(\bar s - 2 \frac{\bar s^T d}{\norm d^2} d\biggr)} < m(\bar s)$.
\item Using the definition of $z$, we can write
    \[
    \begin{split}
    z^T (Q + \sigma \norm{\bar s} I) z & = (\bar s + \alpha d)^T (Q + \sigma \norm{\bar s} I) (\bar s + \alpha d) \\
                                       & = \bar s^T (Q + \sigma \norm{\bar s} I) \bar s + \alpha^2 d^T (Q + \sigma \norm{\bar s} I) d
                                           + 2 \alpha d^T (Q + \sigma \norm{\bar s} I) \bar s.
    \end{split}
    \]
    From stationary conditions~\eqref{stat1}--\eqref{stat2}, we have that $Q \bar s + \sigma \norm{\bar s} \bar s = -c$.
    So, we obtain
    \[
    z^T (Q + \sigma \norm{\bar s} I) z = \alpha^2 d^T (Q + \sigma \norm{\bar s} I) d - 2 \alpha c^T d - c^T \bar s.
    \]
    It is straightforward to verify that the right-hand side of the above equality is negative for all $\alpha > \tilde \alpha$, where
    \[
    \tilde \alpha = \frac{c^T d - \sqrt{(c^T d)^2 + (c^T \bar s) \bigl[d^T (Q + \sigma \norm{\bar s} I) d\bigr]}} {d^T (Q + \sigma \norm{\bar s} I) d}.
    \]
    Consequently, since $z = \bar s + \alpha d$ with $\alpha > \tilde \alpha$, it follows that $z^T (Q + \sigma \norm{\bar s} I) z < 0$.
    We can thus proceed as in case~(ii) by defining the point $\hat s = \bar s - 2 \dfrac{\bar s^T z}{\norm z^2} z$ and we get the result.
\end{enumerate}
\end{proof}

\begin{remark}\label{rem:global}
Conditions of Theorem~\ref{th:global} are satisfied if and only if the stationary point $\bar s$ is not a global minimizer.
It follows from the fact that, if (a) or (b) hold at $\bar s$, then $\bar s$ is not a global minimizer; vice versa,
if $\bar s$ is not a global minimizer, then $(Q + \sigma \norm {\bar s} I)$ is not positive semidefinite (see Theorem~\ref{th:opt})  and then (b) holds.
\end{remark}

Now, we show how the above result can be exploited to obtain a global minimizer of~\eqref{cubic_model} by computing a finite number of stationary points.
We first need the following lemma, stating that two stationary points of problem~\eqref{cubic_model} with the same norm produce the same objective value.
\begin{lemma}\label{lemma:stat}
Let $\hat s$ and $\bar s$ be two points satisfying stationary conditions~\eqref{stat1}--\eqref{stat2} with the same $\lambda$. Then,
\[
m(\hat s) = m(\bar s).
\]
\end{lemma}

\begin{proof}
For every pair $(s,\lambda)$ satisfying~\eqref{stat1}--\eqref{stat2}, we can write
\[
\begin{split}
m(s) & = c^T s + \frac 1 2 s^T (-c - \lambda s) + \frac 1 3 \sigma \|s\|^3 \\
     & = \frac 1 2 c^T s - \frac 1 2 \lambda \|s\|^2 + \frac 1 3 \sigma \|s\|^3 = \frac 1 2 c^T s - \frac 1 6 \sigma \|s\|^3.
\end{split}
\]
Then,
\[
\begin{split}
m(\hat s) & = \frac 1 2 c^T \hat s - \frac 1 6 \sigma \|\hat s\|^3 = -\frac 1 2 \bar s^T (Q + \lambda I) \hat s - \frac 1 6 \sigma \|\bar s\|^3 \\
          & = \frac 1 2 c^T \bar s - \frac 1 6 \sigma \|\bar s\|^3 = m(\bar s).
\end{split}
\]
\end{proof}

The following proposition establishes a bound on the maximum number of stationary points with different norm.
The proof follows the same line of arguments used in~\cite{cartis:2011a} to characterize global minimizers of the cubic model.
It is entirely reported here for the sake of completeness.
\begin{proposition}\label{prop:dist_stat}
At most $2(k+1)$ points that satisfy~\eqref{stat1}--\eqref{stat2} with distinct values of $\lambda$ exist,
where $k$ is the number of distinct negative eigenvalues of $Q$.
\end{proposition}

\begin{proof}
First, we observe that if $\lambda = 0$, then $s = 0$ is the only point that satisfies~\eqref{stat1}--\eqref{stat2}.
So, in the following we consider the case in which $\lambda > 0$ (i.e., $s \ne 0$).
Let $V \in \R^{n \times n}$ be an orthonormal matrix such that
\[
V^T Q V = M,
\]
where $M :=  \text{diag}_{i = 1,\dots,n}\{\mu_i\}$ and $\mu_1 \le \ldots \le \mu_n$ are the eigenvalues of $Q$.
Now, we can introduce the vector $a \in \R^n$ and consider the transformation
\[
s = V a.
\]
Pre-multiplying~\eqref{stat1} by $V^T$, we get
\[
V^T(Q + \lambda I)s = -V^T c,
\]
and then
\[
(M + \lambda I)a = - \beta,
\]
where $\beta = -V^T c$.

\noindent
The above expression can be equivalently written as
\begin{equation}\label{proof:alpha}
a_i = - \frac{\beta_i}{\mu_i + \lambda}, \quad i = 1,\dots,n.
\end{equation}
Moreover, from~\eqref{stat2} we get
\begin{equation}\label{proof:lambda}
\lambda^2 = \sigma^2 \norm s^2 = \sigma^2 \norm{V a}^2 = \sigma^2 \norm a^2.
\end{equation}
Using~\eqref{proof:alpha} and~\eqref{proof:lambda}, we can rewrite the stationary conditions as follows:
\begin{equation}\label{stat_system}
\begin{cases}
g(\lambda) = \dfrac{1}{\sigma^2}, \\
\lambda > 0,
\end{cases}
\end{equation}
where
\[
g(\lambda) := \frac{1}{\lambda^2} \sum_{i=1}^n \frac{\beta_i^2}{(\mu_i + \lambda)^2}.
\]
Now, we have two cases.
\begin{enumerate}[(i)]
\item $\beta_i = 0$ for all $i = 1,\ldots,n$ (i. e., $c=0$). It follows that $g(\lambda) = 0$ in all the domain and
system~\eqref{stat_system} does not admit solutions. In this case, only $s = 0$ satisfies stationary conditions~\eqref{stat1}--\eqref{stat2}.
\item An index $i \in \{1,\ldots,n\}$ exists such that $\beta_i \ne 0$ (i. e., $c \ne 0$).
Without loss of generality, we assume that $\mu_1,\ldots,\mu_p \le 0$, with $p \le n$.
Then $g(\lambda)$ is defined in the following $n+2$ subintervals:
\[
\begin{split}
& (-\infty, -\mu_n) \; \cup \; (-\mu_n ,-\mu_{n-1}) \; \cup \; \ldots \; \cup \; (-\mu_{p+1}, 0) \; \cup \\
& (0, -\mu_p) \; \cup \; \ldots \; \cup \; (-\mu_2 ,-\mu_1) \; \cup \; (-\mu_1, +\infty).
\end{split}
\]
\noindent
Computing the derivatives of $g(\lambda)$, we obtain
\begin{gather*}
\frac{d}{d \lambda} g(\lambda) = - 2 \sum_{i=1}^n \beta_i^2 [\lambda (\mu_i + \lambda)]^{-3} (\mu_i + 2\lambda), \\
\frac{d^2}{d \lambda^2} g(\lambda) = 2 \sum_{i=1}^n \beta_i^2 [\lambda (\mu_i + \lambda)]^{-4} \bigl[10 \lambda^2 + 10 \mu_i \lambda + 3 \mu_i^2].
\end{gather*}
It is straightforward to verify that $\dfrac{d^2}{d \lambda^2} g(\lambda) > 0$ in all the points where $g(\lambda)$ is defined, that is,
$g(\lambda)$ is strictly convex in all the non-empty subintervals that define its domain.

\noindent
Taking into account that $\displaystyle{\lim_{\lambda \to 0} g(\lambda) = +\infty}$,
$\displaystyle{\lim_{\lambda \to -\mu_i} g(\lambda) = +\infty}$ for all $\beta_i \ne 0$
and $\displaystyle{\lim_{\lambda \to \pm \infty} g(\lambda) = 0}$,
we get that $g(\lambda)$ has at most $2(n+1)$ roots: at most one in each extreme subinterval and at most two in all the other subintervals.

\noindent
Now, let $k \le p$ be the number of distinct negative eigenvalues $\mu_i$.
It follows that system~\eqref{stat_system} has at most $2k+1$ solutions:
at most two in each subinterval $(0, -\mu_k)$, $(-\mu_k, -\mu_{k-1})$, $\dots$, $(-\mu_2 ,-\mu_1)$, and at most one in the subinterval $(-\mu_1,+\infty)$.
Taking into account the case $\lambda = 0$, we conclude that there exist at most $2(k+1)$ distinct values of $\lambda$ satisfying stationary conditions~\eqref{stat1}--\eqref{stat2}.
\end{enumerate}
\end{proof}

From Lemma~\ref{lemma:stat} and Proposition~\ref{prop:dist_stat}, we easily get the following corollary,
establishing a bound on the maximum number of distinct values assumed by the objective function $m(s)$ at stationary points.
\begin{corollary}\label{cor:dist_stat_obj}
The maximum number of distinct values of the objective function $m(s)$ at stationary points is $2(k+1)$, where $k$ is the number of distinct negative eigenvalues of $Q$.
\end{corollary}

At least from a theoretical point of view,
Theorem~\ref{th:global} and Corollary~\ref{cor:dist_stat_obj} suggest a possible iterative strategy to obtain a global minimizer of problem~\eqref{cubic_model}.
Namely, we can compute a stationary point $\bar s$ by some local algorithm and check the conditions of Theorem~\ref{th:global}:
if none of them is satisfied, then $\bar s$ is a global minimizer (see Remark~\ref{rem:global}); otherwise, we get a new point $\hat s$ such that $m(\hat s) < m(\bar s)$
and, starting from $\hat s$, we can compute a new stationary point and iterate.
Corollary~\ref{cor:dist_stat_obj} ensures that this procedure is finite and returns a global minimizer of problem~\eqref{cubic_model}.

To be rigorous, the above strategy is well defined under the assumption that stationary points can be computed in a finite number of iterations by a local algorithm.
Unfortunately, optimization methods only ensure asymptotic convergence and, in practice,
a point $\bar s$ is returned such that $\norm{\nabla m(\bar s)} \le \epsilon$, being $\epsilon$ a desired tolerance.
In the next section, we show how Theorem~\ref{th:global} can be generalized to cope with this case and discuss possible algorithmic applications.

\section{Extension to approximate stationary points}\label{sec:practical_usage}
In this section, first we extend Theorem~\ref{th:global} to the case where stationary conditions are approximately satisfied,
and then we briefly discuss how these results may be used in an algorithmic framework, showing some numerical examples.

Assuming that $\bar s \in \R^n$ is a non-stationary point of problem~\eqref{cubic_model}, of course we have $\norm{\nabla m(\bar s)} > 0$, or equivalently,
$\abs{\nabla m(\bar s)^T d} > 0$ for some $d \in \R^n$.
The next theorem states some conditions to compute a point $\hat s$ such that $m(\hat s) < m(\bar s)$.

\begin{theorem}\label{th:global_approx}
Given $\bar s \in \R^n$, let us define the point $\hat s$ as follows:
\begin{enumerate}[(a)]
\item if $c^T \bar s > 0$, then
    \[
    \hat s := -\bar s;
    \]
\item if $c^T \bar s \le 0$ and a vector $d \in \R^n$ exists such that $d^T(Q + \sigma \norm{\bar s} I) d < -\epsilon_2 \norm d^2$,
    \begin{enumerate}[(i)]
    \item if $\bar s = 0$  and $\epsilon_2 \ge 0$, then, assuming without loss of generality that $c^T d \le 0$,
        \[
        \hat s := \bar s + \alpha d,
        \]
        with $\displaystyle{0 < \alpha < -\frac{3 \, d^T Q d}{2 \, \sigma \norm d^3}}$;
    \item if $\bar s \ne 0$, $\bar s^T d \ne 0$ and $\epsilon_2 \ge \biggl|\dfrac{\nabla m(\bar s)^T d}{\bar s^T d}\biggr|$, then
        \[
    	\hat s := \bar s - 2 \frac{\bar s^T d}{\norm d^2} d;
    	\]
    \item if $\bar s \ne 0$, $\bar s^T d = 0$ and $\epsilon_2 > \dfrac{\abs{\nabla m(\bar s)^T \bar s}}{\norm{\bar s}^2}$, then,
        assuming without loss of generality that $\nabla m(\bar s)^T d \ge 0$,
        \[
        \hat s := \bar s - 2 \frac{\bar s^T z}{\norm z^2} z,
        \]
    	where $z:= \bar s + \alpha d$ and $\alpha > 0$ is sufficiently large to satisfy
        \[
        z^T(Q + \sigma \norm{\bar s} I)z < -\epsilon_2 \norm z^2.
        \]
    \end{enumerate}
\end{enumerate}
We have that
\[
m(\hat s) < m(\bar s).
\]
\end{theorem}

\begin{proof}
The proof of case~(a) is the same as for Theorem~\ref{th:global}.
Now, we consider case~(b) and distinguish the three subcases.
\begin{enumerate}[(i)]
\item Since we are assuming that $c^T d \le 0$, we can write
    \[
    \begin{split}
    m(\bar s + \alpha d) & = m(\alpha d) = \alpha c^T d + \frac 1 2 \alpha^2 d^T Q d + \frac 1 3 \sigma \alpha^3 \norm d^3 \\
                         & \le \frac 1 2 \alpha^2 d^T Q d + \frac 1 3 \sigma \alpha^3 \norm d^3
    \end{split}
    \]
    and we obtain the result by the same arguments used in the proof of point~(b)-(i) of Theorem~\ref{th:global}.
\item Using~\eqref{proof:case_b_2}, and exploiting the fact that $d^T(Q + \sigma \norm{\bar s} I) d < -\epsilon_2 \norm d^2$, we get
     \[
     \begin{split}
     m\biggl(\bar s - 2 \frac{\bar s^T d}{\norm d^2} d\biggr)
     & < m(\bar s) - \frac 1 2 \biggl( 2 \frac{\bar s^T d}{\norm d^2} \biggr)^2 \epsilon_2 \norm d^2 - 2 \frac{\bar s^T d}{\norm d^2} \nabla m(\bar s)^T d \\
     & \le m(\bar s) - \frac 1 2 \biggl( 2 \frac{\bar s^T d}{\norm d^2} \biggr)^2 \epsilon_2 \norm d^2 + 2 \frac{\abs{\bar s^T d}}{\norm d^2} \abs{\nabla m(\bar s)^T d} \\
     & = m(\bar s) - 2 \frac{\abs{\bar s^T d}}{\norm d^2} \Bigl( \abs{\bar s^T d} \epsilon_2 - \abs{\nabla m(\bar s)^T d} \Bigr) \le m(\bar s),
     \end{split}
     \]
     where the last inequality follows from the fact that $\epsilon_2 \ge \biggl|\dfrac{\nabla m(\bar s)^T d}{\bar s^T d}\biggr|$.
\item Since $d \ne 0$, we can first assume that $\alpha > 0$ is sufficiently large to satisfy $z \ne 0$.
    Replacing $d$ with $z$ in~\eqref{proof:case_b_2}, we obtain
    \[
    m\biggl(\bar s - 2 \frac{\bar s^T z}{\norm z^2} z\biggr) =
    m(\bar s) + \frac 1 2 \biggl( 2 \frac{\bar s^T z}{\norm z^2} \biggr)^2 z^T (Q + \sigma \norm{\bar s} I) z
    - 2 \frac{\bar s^T z}{\norm z^2} \nabla m(\bar s)^T z.
    \]
    Taking into account that $z = \bar s + \alpha d$ and $\bar s^T z = \bar s^T (\bar s + \alpha d) = \norm{\bar s}^2$,
    we can write
    \begin{equation}\label{proof:ineq}
    \begin{split}
    m\biggl(\bar s - 2 \frac{\bar s^T z}{\norm z^2} z\biggr) & =
    m(\bar s) + \frac 1 2 \biggl( 2 \frac{\norm{\bar s}^2}{\norm z^2} \biggr)^2 z^T (Q + \sigma \norm{\bar s} I) z
          - 2 \frac{\norm{\bar s}^2}{\norm z^2} \nabla m(\bar s)^T (\bar s + \alpha d) \\
    & \le m(\bar s) + \frac 1 2 \biggl( 2 \frac{\norm{\bar s}^2}{\norm z^2} \biggr)^2 z^T (Q + \sigma \norm{\bar s} I) z
          - 2 \frac{\norm{\bar s}^2}{\norm z^2} \nabla m(\bar s)^T \bar s \\
    & \le m(\bar s) + 2 \frac{\norm{\bar s}^2}{\norm z^2} \biggl(\norm{\bar s}^2 \, \frac{z^T (Q + \sigma \norm{\bar s} I) z}{\norm z^2}
          + \abs{\nabla m(\bar s)^T \bar s}\biggr),
    \end{split}
    \end{equation}
    where the first inequality follows from the fact that $\nabla m(\bar s)^T d \ge 0$ and $\alpha > 0$.
    Now, let us define $\theta \in (0,1)$ such that $\epsilon_2 = \dfrac 1 {\theta} \dfrac{\abs{\nabla m(\bar s)^T \bar s}}{\norm{\bar s}^2}$.
    Exploiting the fact that $\theta \in (0,1)$ and $d^T(Q + \sigma \norm{\bar s} I) d < -\epsilon_2 \norm d^2$, for sufficiently large $\alpha > 0$ we have
    \[
    \frac{\biggl(\dfrac{\bar s}{\alpha}+d\biggr)^T (Q + \sigma \norm{\bar s} I) \biggl(\dfrac{\bar s}{\alpha}+d\biggr)}{\dfrac{\norm{\bar s}^2}{\alpha^2} + \norm d^2}
    = \frac{(\bar s + \alpha d)^T (Q + \sigma \norm{\bar s} I) (\bar s + \alpha d)}{\norm{\bar s}^2 + \alpha^2 \norm d^2} < -\theta \epsilon_2.
    \]
    Taking into account that $z = \bar s + \alpha d$ and $\norm z^2 = \norm{\bar s}^2 + \alpha^2 \norm d^2$,
    it follows that, for sufficiently large $\alpha>0$,
    \[
    \begin{split}
    \frac{z^T (Q + \sigma \norm{\bar s} I) z}{\norm z^2} < -\theta \epsilon_2.
    \end{split}
    \]
    Combining this inequality with~\eqref{proof:ineq}, for sufficiently large $\alpha > 0$ we can write
    \[
    m\biggl(\bar s - 2 \frac{\bar s^T z}{\norm z^2} z\biggr) <
    m(\bar s) + 2 \frac{\norm{\bar s}^2}{\norm z^2} \Bigl(-\theta \epsilon_2 \norm{\bar s}^2 + \abs{\nabla m(\bar s)^T \bar s}\Bigr) = m(\bar s),
    \]
    where the equality follows from the fact that $\epsilon_2 = \dfrac 1 {\theta} \dfrac{\abs{\nabla m(\bar s)^T \bar s}}{\norm{\bar s}^2}$.
\end{enumerate}
\end{proof}

\begin{remark}
It is straightforward to verify that, when $\bar s$ is a stationary point,
Theorem~\ref{th:global_approx} coincides with Theorem~\ref{th:global}.
\end{remark}

\begin{remark}
Using~\eqref{proof:case_b_2}, Theorem~\ref{th:global_approx} can be strengthened by replacing the condition~b-(ii) with the condition that
a direction $d$ exists such that $\bar s \ne 0$, $\bar s^T d \ne 0$ and
\[
\frac 1 2 \biggl( 2 \frac{\bar s^T d}{\norm d^2} \biggr)^2 d^T (Q + \sigma \norm{\bar s} I) d - 2 \frac{\bar s^T d}{\norm d^2} \nabla m(\bar s)^T d < 0.
\]
\end{remark}

\begin{remark}
From a computational point of view, condition~(a) of Theorem~\ref{th:global_approx} can be easily checked with a negligible cost.
To check condition~(b), we have to verify if there exists a negative curvature direction with respect to the matrix $(Q + \sigma \norm{\bar s} I)$.
This can be done, for example, by calculating the smallest eigenvalue and the associate eigenvector of that matrix.
If such a direction exists, we see that, for case~(b)-(i), this is enough to ensure that $m(\hat s) < m(\bar s)$.
For case~(b)-(ii) and (b)-(iii), we have to check if $\epsilon_2$ is sufficiently large.
It is easy to verify that, if $\norm{\nabla m(\bar s)} \le \epsilon$, then
condition \mbox{(b)-(ii)} is verified whenever $\epsilon_2 \ge \epsilon \norm d/\abs{\bar s^T d}$,
and condition~(b)-(iii) is verified whenever $\epsilon_2 > \epsilon/\norm{\bar s}$.
Therefore, the threshold value of $\epsilon_2$ for satisfying conditions b-(ii) and \mbox{b-(iii)}
is related to $\norm{\nabla m(\bar s)}$, that is, the tolerance we have chosen to solve problem~\eqref{cubic_model}.
\end{remark}

Let us concluding this section by discussing some possible algorithmic applications of our results,
even if defining a proper optimization method is beyond the scope of the paper.
A first naive strategy to exploit Theorem~\ref{th:global_approx} is checking if one of its conditions holds
after that an approximate stationary point $\bar s$ of problem~\eqref{cubic_model} is computed with the desired tolerance by a local algorithm.
If this is the case, then we can compute the point $\hat s$ and restart the local algorithm from $\hat s$.
To provide some numerical examples, we have inserted this strategy within the ARC algorithm described in~\cite{cartis:2011a,cartis:2011b}
to minimize the cubic model at each iteration, giving rise to an algorithm that we name ARC$^+$.
In particular, at every iteration of ARC$^+$ and ARC, a truncated-Newton method has been used as local solver for the minimization
of the cubic model, starting from a randomly chosen point. The codes have been written in Matlab, using built-in functions to compute eigenvalues and eigenvectors
needed to check the conditions of Theorem~\ref{th:global_approx}.
We have considered a set of $130$ unconstrained test problems of the form $\min_{x \in \R^n} \, f(x)$ from the CUTEst collection~\cite{gould:2015} and,
among them, we have then selected the $39$ for which the two algorithms performed differently and both converged to a point $x^*$ such that
$\norm{\nabla f(x^*)}_{\infty} \le 10^{-5}$ within a maximum number of iterations, set equal to $10^5$.
The results on this subset of problems are reported in Table~\ref{tab:results}, where \textit{obj} and \textit{iter} denote the final objective value
and the number of iterations, respectively.
We see that, in $28$ out $39$ cases, ARC$^+$ converged in fewer iterations.
Taking a look to the performance profile~\cite{dolan:2002} reported in Figure~\ref{fig:performance_profile}, we also observe that, on the considered subset of problems,
ARC$^+$ is more robust than ARC in terms of number of iterations.
We have then repeated the same experiments by using the Cauchy point as starting point for the minimization of the cubic model,
but no significative difference emerged between ARC$^+$ and ARC. This opens a question about possible relations between the Cauchy point
and the global minimizers, which can be subject of future research.

\begin{table}
\centering
\caption{Numerical results of ARC$^+$ and ARC on CUTEst problems. ARC$^+$ differs from ARC in that a globalization strategy,
outlined in Theorem~\ref{th:global_approx}, is used to minimize the cubic model at each iteration. For each problem, the smallest number of iterations is highlighted in bold.}
\bigskip\par
{\begin{tabular}{c c | c c | c c}
\hline
\multirow{2}*{Problem} & \multirow{2}*{$n$} & \multicolumn{2}{c|}{ARC$^+$} & \multicolumn{2}{c}{ARC} \bigstrut[t] \\
                       &                    & obj & iter                   & obj & iter \bigstrut[b] \\
\hline
BROWNAL    & $200$ & $1.00$e$-07$ & $\textbf{103}$ & $1.00$e$-07$ & $472$ \bigstrut[t] \\
BROWNBS    & $2$ & $7.40$e$-12$ & $\textbf{27552}$ & $0.00$e$+00$ & $27560$ \\
CURLY10    & $100$ & $-1.00$e$+04$ & $\textbf{82}$ & $-1.00$e$+04$ & $281$ \\
CURLY20    & $100$ & $-1.00$e$+04$ & $\textbf{53}$ & $-1.00$e$+04$ & $288$ \\
CURLY30    & $100$ & $-1.00$e$+04$ & $\textbf{39}$ & $-1.00$e$+04$ & $590$ \\
DECONVU    & $63$ & $9.10$e$-07$ & $\textbf{162}$ & $8.52$e$-07$ & $167$ \\
DENSCHND   & $3$ & $2.63$e$-07$ & $\textbf{2154}$ & $2.82$e$-07$ & $2293$ \\
DIXMAANH   & $300$ & $1.00$e$+00$ & $424$ & $1.00$e$+00$ & $\textbf{423}$ \\
DIXMAANJ   & $300$ & $1.00$e$+00$ & $4762$ & $1.00$e$+00$ & $\textbf{4739}$ \\
DIXMAANK   & $300$ & $1.00$e$+00$ & $5335$ & $1.00$e$+00$ & $\textbf{5265}$ \\
DIXMAANL   & $300$ & $1.00$e$+00$ & $5008$ & $1.00$e$+00$ & $\textbf{4941}$ \\
EIGENCLS   & $462$ & $4.70$e$-09$ & $\textbf{254}$ & $4.37$e$-09$ & $258$ \\
ENGVAL2    & $3$ & $8.49$e$-16$ & $\textbf{30}$ & $2.04$e$-20$ & $50$ \\
FLETCHBV   & $10$ & $-2.04$e$+06$ & $551$ & $-2.09$e$+06$ & $\textbf{460}$ \\
GENHUMPS   & $10$ & $4.49$e$-12$ & $\textbf{8968}$ & $2.77$e$-11$ & $9283$ \\
GENROSE    & $100$ & $1.00$e$+00$ & $\textbf{119}$ & $1.00$e$+00$ & $120$ \\
GENROSEB   & $500$ & $1.00$e$+00$ & $\textbf{505}$ & $1.00$e$+00$ & $511$ \\
GROWTHLS   & $3$ & $1.00$e$+00$ & $\textbf{271}$ & $1.00$e$+00$ & $4557$ \\
GULF       & $3$ & $3.51$e$-06$ & $4642$ & $3.51$e$-06$ & $\textbf{4640}$ \\
HAIRY      & $2$ & $2.00$e$+01$ & $\textbf{108}$ & $2.00$e$+01$ & $158$ \\
HEART8LS   & $8$ & $4.91$e$-12$ & $\textbf{86}$ & $6.97$e$-17$ & $130$ \\
HUMPS      & $2$ & $1.91$e$-10$ & $\textbf{1611}$ & $8.40$e$-11$ & $1858$ \\
JENSMP     & $2$ & $1.24$e$+02$ & $\textbf{28}$ & $1.24$e$+02$ & $47$ \\
LIARWHD    & $100$ & $1.39$e$-19$ & $\textbf{12}$ & $2.97$e$-20$ & $14$ \\
LOGHAIRY   & $2$ & $1.82$e$-01$ & $\textbf{5177}$ & $1.82$e$-01$ & $5316$ \\
MEXHAT     & $2$ & $-4.00$e$-02$ & $523$ & $-4.00$e$-02$ & $\textbf{68}$ \\
NONCVXU2   & $100$ & $2.33$e$+02$ & $\textbf{571}$ & $2.33$e$+02$ & $572$ \\
NONDIA     & $100$ & $1.57$e$-18$ & $\textbf{7}$ & $9.66$e$-26$ & $9$ \\
OSCIPATH   & $10$ & $1.00$e$+00$ & $39$ & $1.00$e$+00$ & $\textbf{22}$ \\
PALMER6C   & $8$ & $1.64$e$-02$ & $21678$ & $1.64$e$-02$ & $\textbf{17418}$ \\
PALMER7C   & $8$ & $6.02$e$-01$ & $31863$ & $6.02$e$-01$ & $\textbf{24683}$ \\
PALMER8C   & $8$ & $1.60$e$-01$ & $33434$ & $1.60$e$-01$ & $\textbf{14945}$ \\
PARKCH     & $15$ & $1.62$e$+03$ & $\textbf{65}$ & $1.62$e$+03$ & $250$ \\
PFIT1LS    & $3$ & $2.10$e$-10$ & $\textbf{501}$ & $4.75$e$-04$ & $2810$ \\
SINEVAL    & $2$ & $2.13$e$-17$ & $\textbf{101}$ & $5.40$e$-12$ & $137$ \\
SPARSINE   & $100$ & $1.83$e$-14$ & $\textbf{38}$ & $1.13$e$-10$ & $39$ \\
SROSENBR   & $100$ & $4.02$e$-14$ & $\textbf{12}$ & $2.13$e$-17$ & $500$ \\
VARDIM     & $200$ & $6.90$e$-31$ & $\textbf{36}$ & $7.29$e$-27$ & $37$ \\
WATSON     & $12$ & $3.57$e$-06$ & $\textbf{82}$ & $2.84$e$-06$ & $91$ \bigstrut[b] \\
\hline
\end{tabular}}
\label{tab:results}
\end{table}

\begin{figure}
\centering
\includegraphics[scale=0.75]{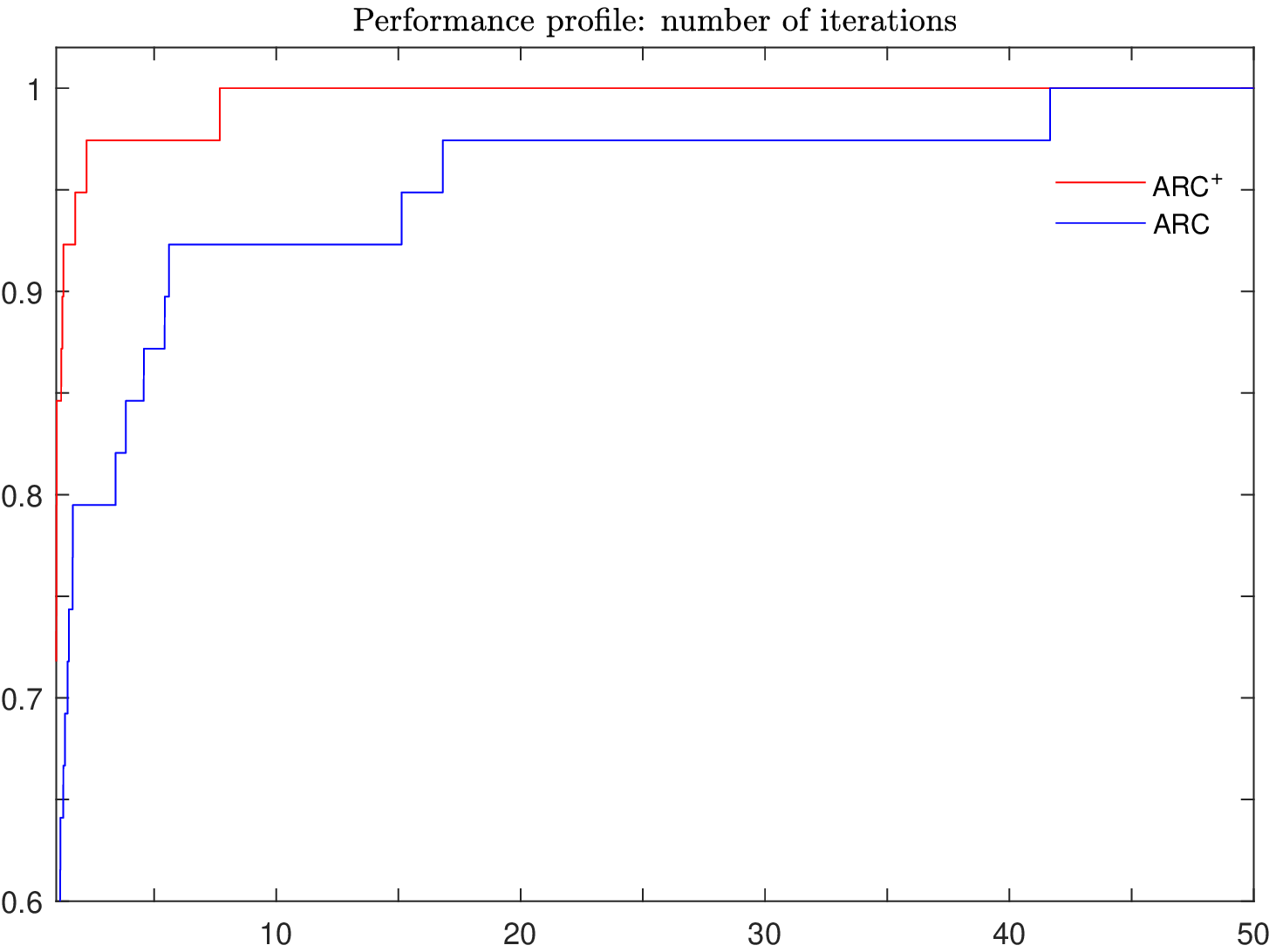}
\caption{Performance profile for the number of iterations related to the numerical experiments reported in Table~\ref{tab:results}.}
\label{fig:performance_profile}
\end{figure}

It is worth pointing out that the above described ARC$^+$ method could be too expensive in terms of CPU time,
since it requires the computation of eigenvalues and eigenvectors at the end of each local minimization.
Nevertheless, a more refined way to exploit Theorem~\ref{th:global_approx} for algorithmic purposes can be based on checking if
one of its conditions is satisfied during the iterations of the local method, instead of at the end.
This can be done efficiently when the local method is able to detect negative curvature directions.
Assuming that a sequence of points $\{s^k\}$ and a sequence of directions $\{d^k\}$ are produced by the local algorithm,
since $\nabla^2 m(s^k) = Q + \sigma \norm{s^k} I + \sigma \dfrac{s^k (s^k)^T}{\norm{s^k}}$, we have
$\displaystyle{(d^k)^T (Q + \sigma \norm{s^k} I) d^k = (d^k)^T \nabla^2 m(s^k) d^k - \sigma \frac{((s^k)^T d^k)^2}{\norm{s^k}}}$.
Therefore, if $d^k$ is a negative curvature direction with respect to $\nabla^2 m(s^k)$, condition~(b) of Theorem~\ref{th:global_approx}
is verified for some $\epsilon_2 \ge 0$, provided $c^T \bar s \le 0$. Then, a new point that ensures a decrease in the objective function may be easily computed.
In this case, condition~(b) of Theorem~\ref{th:global_approx} can therefore be checked without the need of computing eigenvalues and eigenvectors.
Finally, other checks can be included in the scheme to ensure convergence of such modification of the local algorithm.

\section{Conclusions}\label{sec:conclusions}
In this paper, we have highlighted some theoretical properties of the stationary points of problem~\eqref{cubic_model},
whose solutions are of interest for many optimization methods.
We have shown that, given a stationary point of problem~\eqref{cubic_model} that is not a global minimizer,
it is possible to compute, in closed form, a new point that reduces the objective function value.
Then, we have pointed out how a global minimum point
of problem~\eqref{cubic_model} can be obtained by computing at most $2(k+1)$ stationary points, where $k$ is the number
of distinct negative eigenvalues of the matrix~$Q$.
Further, we have extended these results to the case where stationary conditions are approximately satisfied,
sketching some possible algorithmic applications.

We think that the most natural extension of the results presented in this paper is the definition of a proper algorithm
for unconstrained optimization, based on the iterative computation of the solutions of problem~\eqref{cubic_model},
for which some preliminary ideas have been proposed at the end of Section~\ref{sec:practical_usage}.
This can be a challenging task for future research.


\bibliographystyle{spbasic}      
\bibliography{cubic}   

\begin{thebibliography}{19}
\providecommand{\natexlab}[1]{#1}
\providecommand{\url}[1]{\texttt{#1}}
\expandafter\ifx\csname urlstyle\endcsname\relax
  \providecommand{\doi}[1]{doi: #1}\else
  \providecommand{\doi}{doi: \begingroup \urlstyle{rm}\Url}\fi

\bibitem[Bellavia and Morini(2014)]{bellavia:2014}
S.~Bellavia and B.~Morini.
\newblock {Strong local convergence properties of adaptive regularized methods
  for nonlinear least squares}.
\newblock \emph{IMA Journal of Numerical Analysis}, 35\penalty0 (2):\penalty0
  947--968, 2014.

\bibitem[Benson and Shanno(2014)]{benson:2014}
H.~Y. Benson and D.~F. Shanno.
\newblock {Interior-point methods for nonconvex nonlinear programming: cubic
  regularization}.
\newblock \emph{Computational Optimization and Applications}, 58\penalty0
  (2):\penalty0 323--346, 2014.

\bibitem[Bianconcini and Sciandrone(2016)]{bianconcini:2016}
T.~Bianconcini and M.~Sciandrone.
\newblock {A cubic regularization algorithm for unconstrained optimization
  using line search and nonmonotone techniques}.
\newblock \emph{Optimization Methods and Software}, 31\penalty0 (5):\penalty0
  1008--1035, 2016.

\bibitem[Bianconcini et~al.(2015)Bianconcini, Liuzzi, Morini, and
  Sciandrone]{bianconcini:2015}
T.~Bianconcini, G.~Liuzzi, B.~Morini, and M.~Sciandrone.
\newblock {On the use of iterative methods in cubic regularization for
  unconstrained optimization}.
\newblock \emph{Computational Optimization and Applications}, 60\penalty0
  (1):\penalty0 35--57, 2015.

\bibitem[Birgin et~al.(2017)Birgin, Gardenghi, Mart{\'\i}nez, Santos, and
  Toint]{birgin:2017}
E.~G. Birgin, J.~L. Gardenghi, J.~M. Mart{\'\i}nez, S.~A. Santos, and P.~L.
  Toint.
\newblock {Worst-case evaluation complexity for unconstrained nonlinear
  optimization using high-order regularized models}.
\newblock \emph{Mathematical Programming}, 163\penalty0 (1-2):\penalty0
  359--368, 2017.

\bibitem[Cartis et~al.(2011{\natexlab{a}})Cartis, Gould, and
  Toint]{cartis:2011a}
C.~Cartis, N.~I.~M. Gould, and P.~L. Toint.
\newblock {Adaptive cubic regularisation methods for unconstrained
  optimization. Part I: motivation, convergence and numerical results}.
\newblock \emph{Mathematical Programming}, 127\penalty0 (2):\penalty0 245--295,
  2011{\natexlab{a}}.

\bibitem[Cartis et~al.(2011{\natexlab{b}})Cartis, Gould, and
  Toint]{cartis:2011b}
C.~Cartis, N.~I.~M. Gould, and P.~L. Toint.
\newblock {Adaptive cubic regularisation methods for unconstrained
  optimization. Part II: worst-case function-and derivative-evaluation
  complexity}.
\newblock \emph{Mathematical programming}, 130\penalty0 (2):\penalty0 295--319,
  2011{\natexlab{b}}.

\bibitem[Cartis et~al.(2012)Cartis, Gould, and Toint]{cartis:2012}
C.~Cartis, N.~I.~M. Gould, and P.~L. Toint.
\newblock {An adaptive cubic regularization algorithm for nonconvex
  optimization with convex constraints and its function-evaluation complexity}.
\newblock \emph{IMA Journal of Numerical Analysis}, 32\penalty0 (4):\penalty0
  1662--1695, 2012.

\bibitem[Conn et~al.(2000)Conn, Gould, and Toint]{conn:2000}
A.~R. Conn, N.~I.~M. Gould, and P.~L. Toint.
\newblock \emph{{Trust region methods}}.
\newblock SIAM, 2000.

\bibitem[Dolan and Mor{\'e}(2002)]{dolan:2002}
E.~D. Dolan and J.~J. Mor{\'e}.
\newblock {Benchmarking optimization software with performance profiles}.
\newblock \emph{Mathematical programming}, 91\penalty0 (2):\penalty0 201--213,
  2002.

\bibitem[Dussault(2015)]{dussault:2015}
J.-P. Dussault.
\newblock {Simple unified convergence proofs for the trust-region and a new ARC
  variant}.
\newblock Technical report, University of Sherbrooke, Sherbrooke, Canada, 2015.

\bibitem[Gould et~al.(2012)Gould, Porcelli, and Toint]{gould:2012}
N.~I.~M. Gould, M.~Porcelli, and P.~L. Toint.
\newblock {Updating the regularization parameter in the adaptive cubic
  regularization algorithm}.
\newblock \emph{Computational Optimization and Applications}, 53\penalty0
  (1):\penalty0 1--22, 2012.

\bibitem[Gould et~al.(2015)Gould, Orban, and Toint]{gould:2015}
N.~I.~M. Gould, D.~Orban, and P.~L. Toint.
\newblock {CUTEst: a constrained and unconstrained testing environment with
  safe threads for mathematical optimization}.
\newblock \emph{Computational Optimization and Applications}, 60\penalty0
  (3):\penalty0 545--557, 2015.

\bibitem[Griewank(1981)]{griewank:1981}
A.~Griewank.
\newblock {The modification of Newton’s method for unconstrained optimization
  by bounding cubic terms}.
\newblock Technical Report NA/12, 1981.

\bibitem[Lucidi et~al.(1998)Lucidi, Palagi, and Roma]{lucidi:1998}
S.~Lucidi, L.~Palagi, and M.~Roma.
\newblock {On some properties of quadratic programs with a convex quadratic
  constraint}.
\newblock \emph{SIAM Journal on Optimization}, 8\penalty0 (1):\penalty0
  105--122, 1998.

\bibitem[Nesterov(2006)]{nesterov:2006b}
Y.~Nesterov.
\newblock {Cubic regularization of Newton's method for convex problems with
  constraints}.
\newblock Technical Report~39, CORE, 2006.

\bibitem[Nesterov(2008)]{nesterov:2008}
Y.~Nesterov.
\newblock {Accelerating the cubic regularization of Newton’s method on convex
  problems}.
\newblock \emph{Mathematical Programming}, 112\penalty0 (1):\penalty0 159--181,
  2008.

\bibitem[Nesterov and Polyak(2006)]{nesterov:2006a}
Y.~Nesterov and B.~T. Polyak.
\newblock {Cubic regularization of Newton method and its global performance}.
\newblock \emph{Mathematical Programming}, 108\penalty0 (1):\penalty0 177--205,
  2006.

\bibitem[Weiser et~al.(2007)Weiser, Deuflhard, and Erdmann]{weiser:2007}
M.~Weiser, P.~Deuflhard, and B.~Erdmann.
\newblock {Affine conjugate adaptive Newton methods for nonlinear
  elastomechanics}.
\newblock \emph{Optimization Methods and Software}, 22\penalty0 (3):\penalty0
  413--431, 2007.

\end{thebibliography}

\end{document}